\documentclass[12pt]{article}
\usepackage{amssymb}
\usepackage{amsfonts}
\usepackage{bbm}
\usepackage{amsmath,amssymb,amscd,amsthm,amsfonts,amsthm}
\numberwithin{equation}{section}
\newtheorem{thm}{Theorem}[section]
\newtheorem{cor}[thm]{Corollary}
\newtheorem{lem}[thm]{Lemma}

\date{}
\title{{\bf A proof of the DDVV conjecture and its equality case}
\thanks{The project is partially supported by the NSFC ( No.10531090 and No.10229101 ) and the Chang Jiang Scholars
Program.}}
\author{{Ge Jianquan}\\
{\small Department of Mathematical Sciences, Tsinghua University,}\\
{\small Beijing 100084.}\\
{\small E-mail: gejq04@mails.tsinghua.edu.cn}\\
{Tang Zizhou\thanks{The corresponding author.}}\\
{\small School of Mathematical Sciences, Laboratory of Mathematics and Complex Systems,}\\
{\small Beijing Normal University, Beijing 100875.}\\
{\small E-mail: zztang@mx.cei.gov.cn}}
\begin{document} \maketitle
\begin{abstract}
In this paper, we give a proof of the DDVV conjecture which is a
pointwise inequality involving the scalar curvature, the normal
scalar curvature and the mean curvature on a submanifold of a real
space form. Furthermore we solved the problem of its equality case.
\end{abstract}
\section{Introduction}
  \hskip 0.5cm Let $f: M^n\rightarrow N^{n+m}(c)$ be an isometric
immersion of an $n$-dimensional submanifold $M$ into the
$(n+m)$-dimensional real space form $N^{n+m}(c)$ of constant
sectional curvature $c$. The normalized scalar curvature $\rho$ and
normal scalar curvature $\rho^{\perp}$ are defined by (see [DDVV])
\[\rho=\frac{2}{n(n-1)}\sum^n_{1=i<j}R(e_i,e_j,e_j,e_i),\]
\[\rho^{\perp}=\frac{2}{n(n-1)}(\sum^n_{1=i<j}\sum^m_{1=r<s}\langle R^{\perp}(e_i,e_j)\xi_r, \xi_s\rangle^2)^{\frac{1}{2}},\]
where $\{e_1,...,e_n\}$ (resp. $\{\xi_1,...,\xi_m\}$) is an
orthonormal basis of the tangent (resp. normal) space, and $R$
(resp. $R^{\perp}$) is the curvature tensor of the tangent (resp.
normal) bundle. \vskip 0.05cm Let $h$ be the second fundamental form
and let $H=\frac{1}{n}\hskip 0.2cm Tr\hskip 0.1cm h$ be the mean
curvature vector field. The DDVV conjecture (see [DDVV]) says that
there's a pointwise inequality among $\rho$, $\rho^{\perp}$ and
$|H|^2$ as following:
\[\rho+\rho^{\perp}\leq |H|^2+c.\]
\vskip 0.05cm Since this is a pointwise inequality, using Gauss and
Ricci equations one can see that it's equivalent to the following
algebraic inequality (see [DFV]):\\
\begin{bf}
 Conjecture 1.
 \end{bf}
Let $B_1,...,B_m$ be ($n\times n$) real symmetric matrices. Then
\[\sum_{r,s=1}^m\|[B_r,B_s]\|^2\leq (\sum_{r=1}^m\|B_r\|^2)^2,\]
where $\|\cdot\|^2$ denotes the sum of squares of entries of the
matrix and $[A,B]=AB-BA$ is the commutator of the matrices $A,B$.
\vskip 0.05cm The main purpose of this paper is to prove Conjecture
1 and also to give the equality condition:
\begin{thm}
Let $B_1,...,B_m$ be ($n\times n$) real symmetric matrices. Then
\[\sum_{r,s=1}^m\|[B_r,B_s]\|^2\leq (\sum_{r=1}^m\|B_r\|^2)^2,\]
where the equality holds if and only if under some
rotation\footnote{An orthogonal $m\times m$ matrix $R=(R_{rs})$ acts
as a rotation on $(B_1,...,B_m)$ by $R(B_r)=\sum_{s=1}^mR_{sr}B_s$.}
all $B_r$'s are zero except 2 matrices which can be written as
\[P\left(\begin{array}{ccccc}0& \mu& 0&\cdots& 0\\\mu& 0& 0&\cdots&
0\\0& 0& 0&\cdots& 0\\ \vdots&\vdots & \vdots&\ddots &\vdots
\\0& 0& 0&\cdots& 0 \end{array}\right)P^t, \hskip 0.5cm P\left(\begin{array}{ccccc}\mu& 0& 0&\cdots& 0\\0& -\mu& 0&\cdots&
0\\0& 0& 0&\cdots& 0\\ \vdots&\vdots & \vdots&\ddots &\vdots
\\0& 0& 0&\cdots& 0 \end{array}\right)P^t,\]
where $P$ is an orthogonal $(n\times n)$ matrix.
\end{thm}
\vskip 0.05cm Therefore, we can solve the DDVV conjecture also with
its equality conditions in terms of the shape operators:
\begin{cor}
Let $f: M^n\rightarrow N^{n+m}(c)$ be an isometric immersion. Then
\[\rho+\rho^{\perp}\leq |H|^2+c,\]
where the equality holds at some point $p\in M$ if and only if there
exist an orthonormal basis $\{e_1,...,e_n\}$ of $T_pM$ and an
orthonormal basis $\{\xi_1,...,\xi_m\}$ of $T_p^{\perp}M$, such that
\[A_{\xi_1}=\left(\begin{array}{ccccc}\lambda_1 & \mu& 0&\cdots&
0\\\mu& \lambda_1& 0&\cdots& 0\\0& 0& \lambda_1&\cdots& 0\\
\vdots&\vdots & \vdots&\ddots &\vdots
\\0& 0& 0&\cdots& \lambda_1 \end{array}\right),\hskip 0.5cm  A_{\xi_2}=\left(\begin{array}{ccccc}\lambda_2+\mu& 0& 0&\cdots& 0\\0& \lambda_2-\mu& 0&\cdots&
0\\0& 0& \lambda_2&\cdots& 0\\ \vdots&\vdots & \vdots&\ddots &\vdots
\\0& 0& 0&\cdots& \lambda_2 \end{array}\right),\]
\[A_{\xi_r}=\left(\begin{array}{ccccc}\lambda_r& 0&
0&\cdots& 0\\0& \lambda_r& 0&\cdots& 0\\0& 0& \lambda_r&\cdots& 0\\
\vdots&\vdots & \vdots&\ddots &\vdots
\\0& 0& 0&\cdots& \lambda_r \end{array}\right)\hskip 0.2cm for\hskip 0.2cm r>2.\]

\end{cor}
\begin{bf}
 Remark.
 \end{bf}
By the same method, one can see that Conjecture 1 also holds for
anti-symmetric matrices. However, the following example shows that
Conjecture 1 fails when there're both symmetric and anti-symmetric
matrices in $\{B_1,...,B_m\}$, which was conjectured in [Lu1].\vskip
0.05cm
\begin{bf}
 Example.
 \end{bf}
 Let $B_1=\left(\begin{array}{cc}1& 0\\0& -1\end{array}\right)$, $B_2=\left(\begin{array}{cc}0& 1\\1&
 0\end{array}\right)$, $B_3=\left(\begin{array}{cc}0& 1\\-1&
 0\end{array}\right)$. Then the conclusion of Conjecture 1 fails.
 \vskip 0.05cm
 We point out that the inequality and its equality condition of Theorem 1.1
 (resp. Corollary 1.2) for $m=2$ was given in [Ch] (resp.
 [DDVV]). When $n=2, 3$, or $m=2, 3$, the inequality was proved in
 several papers (see [Lu1] for references). For general $n$, $m$, a
 weaker version was proved in [DFV]. After we have solved the conjecture and its equality case,
 we find very recently that Zhiqin Lu has posed a proof of the inequality in his
 homepage without the equality case (see [Lu2]). Since we use a
 quite different method and work out the equality condition besides the
 inequality, we'd like to show it in literature. \vskip 0.05cm
 Finally, it's our pleasure to thank Professor Weiping Zhang for
 introducing [Lu1] to us and his encouragements. Many thanks as well
 to Professors Marcos Dajczer and Ruy Tojeiro for their useful
 comments and suggestions to the previous version of this paper.
\section{Notations and preparing lemmas}
\hskip 0.5cm Throughout this paper, we denote by $M(m,n)$ the space
of $m\times n$ real matrices, $M(n)$ the space of $n\times n$ real
matrices and $SM(n)$ the $N:=\frac{n(n+1)}{2}$ dimensional subspace
of symmetric matrices in $M(n)$. \vskip 0.05cm For every $(i,j)$
with $1\leq i\leq j\leq n$, let
\[\hat{E}_{ij}:=\begin{cases}
E_{ii}, \hskip 0.85 in i=j,\\
\frac{1}{\sqrt{2}}(E_{ij}+E_{ji}), \hskip 0.1 in i<j,
\end{cases}\]
where $E_{ij}\in M(n)$ is the matrix with $(i,j)$ entry $1$ and all
others $0$. Clearly $\{\hat{E}_{ij}\}_{i\leq j}$ is an orthonormal
basis of $SM(n)$. Let's take an order of the indices set
$S:=\{(i,j)| 1\leq i\leq j\leq n\}$ by
\begin{equation}
(i,j)<(k,l) \hskip 0.2cm if\hskip 0.1cm and\hskip 0.1cm only\hskip
0.1cm if \hskip 0.2cm i<k\hskip 0.1cm or\hskip 0.1cm i=k, j<l.
\end{equation}
\vskip 0.05cm In this way we can identify $S$ with $\{1,...,N\}$ and
write elements of $S$ in Greek, i.e. for $\alpha=(i,j)\in S$, we can
say $1\leq\alpha\leq N$. \vskip 0.05cm For
$\alpha=(i,j)<(k,l)=\beta$ in $S$, direct calculations imply
\begin{equation}
\|[\hat{E}_{\alpha}, \hat{E}_{\beta}]\|^2=
\begin{cases}
1, \hskip 0.3cm i=j=k<l \hskip 0.1cm or \hskip 0.1cm i<j=k=l;\\
\frac{1}{2}, \hskip 0.3cm i<j=k<l \hskip 0.1cm or \hskip 0.1cm
i=k<j<l \hskip 0.1cm or \hskip 0.1cm i<k<j=l;\\
0, \hskip 0.3cm otherwise,
\end{cases}
\end{equation}
and for any $\alpha, \beta\in S,$
\begin{equation}
\sum_{\gamma\in S}\hskip 0.1cm \langle\hskip 0.1cm
[\hat{E}_{\alpha}, \hat{E}_{\gamma}],\hskip 0.1cm [\hat{E}_{\beta},
\hat{E}_{\gamma}]\hskip 0.1cm \rangle=
n\delta_{\alpha\beta}-\delta_{\alpha}\delta_{\beta},
\end{equation}
where $\delta_{\alpha\beta}=\delta_{ik}\delta_{jl}$,
$\delta_{\alpha}=\delta_{ij}$, $\delta_{\beta}=\delta_{kl}$ and
$\langle \cdot , \cdot\rangle$ is the standard inner product of
$M(n)$. \vskip 0.05cm Let $\{\hat{Q}_{\alpha}\}_{\alpha\in S}$ be
any orthonormal basis of $SM(n)$. There exists a unique orthogonal
matrix $Q\in O(N)$ such that
$(\hat{Q}_1,...,\hat{Q}_N)=(\hat{E}_1,...,\hat{E}_N)Q$, i.e.
$\hat{Q}_{\alpha}=\sum_{\beta}q_{\beta\alpha}\hat{E}_{\beta}$ for
$Q=(q_{\alpha\beta})_{N\times N}$ and if
$\hat{Q}_{\alpha}=(\hat{q}^{\alpha}_{ij})_{n\times n}$,
\[\hat{q}^{\alpha}_{ij}=\hat{q}^{\alpha}_{ji}=
\begin{cases}
q_{\beta\alpha}, \hskip 0.8cm \beta=(i,j)\hskip 0.1cm and \hskip
0.1cm i=j,\\
\frac{1}{\sqrt{2}}q_{\beta\alpha}, \hskip 0.3cm \beta=(i,j)\hskip
0.1cm and \hskip 0.1cm i<j.
\end{cases}
\]
\vskip 0.05cm Let $\lambda_1,...,\lambda_n$ be $n$ real numbers
satisfying $\sum_i\lambda_i^2=1$ and $\lambda_1\geq ...\geq
\lambda_n$. Denote $I_1:=\{j|\lambda_1-\lambda_j>1\}$,
$I_2:=\{i|\lambda_i-\lambda_n>1\}$,
$I:=\{(i,j)|\lambda_i-\lambda_j>1\}$ and $n_0$ the number of
elements of $I$. Then $(\{1\}\times I_1)\cup (I_2\times\{n\})\subset
I\subset S$. In fact, it can be shown
\begin{lem}
We have either
\[I=\{1\}\times I_1 \hskip 0.2cm or \hskip 0.2cm I=I_2\times\{n\}. \]
\end{lem}
\begin{proof}
If $n_0=0$, the three sets are all empty. If $n_0=1$, the only
element must be $(1,n)$ and the three sets are equal. Now let
$(1,n)$, $(i_1,j_1)$ be two different elements of $I$, i.e.
$\lambda_1-\lambda_n\geq \lambda_{i_1}-\lambda_{j_1}>1$ and
$(1,n)\neq (i_1,j_1)$. We assert that $(i_1=1, j_1\neq n)$ or
$(i_1\neq 1, j_1=n)$, which shows exactly that $I=\{1\}\times
I_1\cup I_2\times\{n\}$. Otherwise, $1, i_1, j_1, n$ will be 4
different elements in $\{1,...,n\}$ and thus
\[1\geq \lambda^2_1+\lambda^2_{i_1}+\lambda^2_{j_1}+\lambda^2_n\geq \frac{1}{2}(\lambda_1-\lambda_n)^2+\frac{1}{2}(\lambda_{i_1}-\lambda_{j_1})^2>1\]
is a contradiction. Without loss of generality, we can assume
$(i_1,j_1)\in \{1\}\times I_1$. Then it'll be seen that
$I_2\times\{n\}=\{(1,n)\}$ and thus $I=\{1\}\times I_1$ which
completes the proof. Otherwise, if there's another element, say
$(i_2,n)$ in $I_2\times\{n\}$, then $i_1=1, j_1, i_2, n$ are 4
different elements in $\{1,...,n\}$ and come to the same
contradiction as above.
\end{proof}
\begin{lem}
\[\sum_{(i,j)\in I}[(\lambda_i-\lambda_j)^2-1]\leq 1,\]
where the equality holds in the case when $I=\{1\}\times I_1$ if and
only if $1\leq n_0<n$ and $\lambda_1=\sqrt{\frac{n_0}{n_0+1}},
\lambda_{n-n_0+1}=...=\lambda_{n}=-\frac{1}{\sqrt{n_0^2+n_0}},
\lambda_k=0$ for others.
\end{lem}
\begin{proof}
Without loss of generality, we can assume that $I=\{1\}\times I_1$
by Lemma 2.1. Then
\begin{eqnarray}   \sum_{(i,j)\in I}[(\lambda_i-\lambda_j)^2-1]&=&\sum_{j\in I_1}(\lambda^2_1+\lambda^2_j-2\lambda_1\lambda_j)-n_0
=n_0\lambda^2_1+\sum_{j\in I_1}\lambda^2_j-2\lambda_1\sum_{j\in
I_1}\lambda_j-n_0\nonumber\\
&\leq&(n_0+1)\lambda^2_1+\sum_{j\in I_1}\lambda^2_j+(\sum_{j\in
I_1}\lambda_j)^2-n_0
\leq(n_0+1)(\lambda^2_1+\sum_{j\in I_1}\lambda^2_j)-n_0\nonumber\\
&\leq&(n_0+1)\sum_{i}\lambda^2_i-n_0=1,\nonumber
\end{eqnarray}
where the equality condition is easily seen from the proof.
\end{proof}
\begin{lem}
$\sum_{\beta\in J_{\alpha}}( \|[\hat{Q}_{\alpha},
\hat{Q}_{\beta}]\|^2-1)\leq 1 $ for any $Q\in O(N)$, $\alpha\in S$
and $J_{\alpha}\subset S$.
\end{lem}
\begin{proof}
Given $\alpha\in S$, without loss of generality, we can assume
$\hat{Q}_{\alpha}=diag(\lambda_1,...,\lambda_n)$ with
$\sum_i\lambda^2_i=1$ and $\lambda_1\geq...\geq\lambda_n$. Then by
Lemma 2.2,
\begin{eqnarray}
\sum_{\beta\in J_{\alpha}}( \|[\hat{Q}_{\alpha},
\hat{Q}_{\beta}]\|^2-1)&=&\sum_{\beta\in
J_{\alpha}}\sum^n_{i,j=1}((\lambda_i-\lambda_j)^2-1)(\hat{q}^{\beta}_{ij})^2=\sum_{\beta\in
J_{\alpha}}\sum_{(i,j)=\gamma\in
S}((\lambda_i-\lambda_j)^2-1)q^2_{\gamma\beta}\nonumber\\
&\leq&\sum_{(i,j)=\gamma\in
I}((\lambda_i-\lambda_j)^2-1)\sum_{\beta\in
J_{\alpha}}q^2_{\gamma\beta}\leq\sum_{(i,j)\in
I}((\lambda_i-\lambda_j)^2-1)\leq 1.\nonumber
\end{eqnarray}
\end{proof}
\begin{lem}
$\sum_{\beta\in S}\|[\hat{Q}_{\alpha}, \hat{Q}_{\beta}]\|^2\leq n$
for any $Q\in O(N)$, $\alpha\in S$.
\end{lem}
\begin{proof}
It follows from $(2.3)$ that
\begin{eqnarray}
\sum_{\beta\in S}\|[\hat{Q}_{\alpha},
\hat{Q}_{\beta}]\|^2&=&\sum_{\beta\gamma\tau\xi\eta}q_{\gamma\alpha}q_{\xi\alpha}q_{\tau\beta}q_{\eta\beta}\langle\hskip
0.1cm [\hat{E}_{\gamma}, \hat{E}_{\tau}],\hskip 0.1cm
[\hat{E}_{\xi}, \hat{E}_{\eta}]\hskip 0.1cm \rangle
=\sum_{\gamma\xi}q_{\gamma\alpha}q_{\xi\alpha}\sum_{\tau}\langle\hskip
0.1cm [\hat{E}_{\gamma}, \hat{E}_{\tau}],\hskip 0.1cm
[\hat{E}_{\xi}, \hat{E}_{\tau}]\hskip 0.1cm \rangle\nonumber\\
&=&\sum_{\gamma\xi}q_{\gamma\alpha}q_{\xi\alpha}(n\delta_{\gamma\xi}-\delta_{\gamma}\delta_{\xi})
=n\sum_{\gamma}q^2_{\gamma\alpha}-(\sum_i\hat{q}^{\alpha}_{ii})^2\leq
n.\nonumber
\end{eqnarray}
\end{proof}
\vskip 0.05cm Now let $\varphi : M(m,n)\longrightarrow
M(C_m^2,C_n^2)$ be the map defined by
$\varphi(A)_{(i,j)(k,l)}:=A(_{i\hskip 0.1cm j}^{k\hskip 0.1cm l})$,
where $1\leq i<j\leq m$, $1\leq k<l\leq n$ and $A(_{i\hskip 0.1cm
j}^{k\hskip 0.1cm l})=a_{ik}a_{jl}-a_{il}a_{jk}$ is the discriminant
of the sub-matrix of $A$ with the rows $i, j$, the columns $k, l$,
arranged with the same ordering as in $(2.1)$. It's easily seen that
$\varphi(I_n)=I_{C_n^2}$, $\varphi(A)^t=\varphi(A^t)$ and the
following
\begin{lem} The map $\varphi$ preserves the matrix product, i.e.
 $\varphi(AB)=\varphi(A)\varphi(B)$ holds for $A\in M(m,k)$, $B\in
M(k,n)$.
\end{lem}
\section{Proof of the main results}
\hskip 0.5cm Let $B_1,...,B_m$ be any real symmetric $n\times n$
matrices. Their coefficients under the standard basis
$\{\hat{E}_{\alpha}\}_{\alpha\in S}$ of $SM(n)$ are determined by a
matrix $B\in M(N,m)$ as $(B_1,...,B_m)=(\hat{E}_1,...,\hat{E}_N)B$.
Taking the same ordering as in $(2.1)$ for $1\leq r<s\leq m$ and
$1\leq \alpha<\beta\leq N$, we arrange $\{[B_r, B_s]\}_{r<s}$,
$\{[\hat{E}_{\alpha}, \hat{E}_{\beta}]\}_{\alpha<\beta}$ into
$C_m^2$, $C_N^2$ vectors respectively. We first observe that
\[([B_1, B_2],...,[B_{m-1}, B_m])=([\hat{E}_1, \hat{E}_2],...,[\hat{E}_{N-1}, \hat{E}_N])\cdot \varphi(B).\]
\vskip 0.05cm Let $C(E)$ denote the matrix in $M(C_N^2)$ defined by
$C(E)_{(\alpha,\beta)(\gamma,\tau)}:=\langle\hskip 0.1cm
[\hat{E}_{\alpha}, \hat{E}_{\beta}],\hskip 0.1cm [\hat{E}_{\gamma},
\hat{E}_{\tau}]\hskip 0.1cm \rangle$, for $1\leq\alpha<\beta\leq N$,
$1\leq\gamma<\tau\leq N$. Moreover we will use the same notation for
$\{B_r\}$ and $\{\hat{Q}_{\alpha}\}$, i.e. $C(B)$ and $C(Q)$
respectively. Then it's obvious that
\[C(B)=\varphi(B^t)C(E)\varphi(B), \hskip 0.3cm C(Q)=\varphi(Q^t)C(E)\varphi(Q).\]
Since $BB^t$ is a semi-positive definite matrix in $SM(N)$, there
exists an orthogonal matrix $Q\in SO(N)$ such that $BB^t=Q\hskip
0.1cm diag(x_1,...,x_N)\hskip 0.1cm Q^t$ with $x_{\alpha}\geq 0$,
$1\leq\alpha\leq N.$ Thus
\[\sum_{r=1}^m\|B_r\|^2=\|B\|^2 =\sum_{\alpha=1}^Nx_{\alpha}\] and
hence by Lemma 2.5
\begin{eqnarray}
\sum_{r,s=1}^m\|[B_r, B_s]\|^2&=&2Tr\hskip 0.1cm C(B)=2Tr\hskip
0.1cm \varphi(B^t)C(E)\varphi(B)=2Tr\hskip 0.1cm
\varphi(BB^t)C(E)\nonumber\\&=&2Tr\hskip 0.1cm
\varphi(diag(x_1,...,x_N))C(Q)=\sum_{\alpha,\beta=1}^Nx_{\alpha}x_{\beta}\|[\hat{Q}_{\alpha},
\hat{Q}_{\beta}]\|^2.
\end{eqnarray}
\vskip 0.05cm We are now ready to prove Theorem 1.1. \\
\begin{bf}
 Proof of Theorem 1.1.
 \end{bf}
For the inequality, by the arguments above it is equivalent to prove
\begin{equation}
\sum_{\alpha,\beta=1}^Nx_{\alpha}x_{\beta}\|[\hat{Q}_{\alpha},
\hat{Q}_{\beta}]\|^2\leq (\sum_{\alpha=1}^Nx_{\alpha})^2, \hskip
0.5cm for\hskip 0.1cm any\hskip 0.1cm x\in \mathbb{R}^N_{+} \hskip
0.1cm and \hskip 0.1cm Q\in SO(N),
\end{equation}
 where
$\mathbb{R}^N_{+}:=\{0\neq x=(x_1,...,x_N)\in
\mathbb{R}^N|x_{\alpha}\geq 0, 1\leq\alpha\leq N\}$ is the cone
spanned by the positive axes of $\mathbb{R}^N$. \vskip 0.05cm Let
$f_Q(x)=F(x,Q):=\sum_{\alpha,\beta=1}^Nx_{\alpha}x_{\beta}\|[\hat{Q}_{\alpha},
\hat{Q}_{\beta}]\|^2-(\sum_{\alpha=1}^Nx_{\alpha})^2$. Then $F$ is a
continuous function defined on $\mathbb{R}^N\times SO(N)$ and thus
uniformly continuous on any compact subset of $\mathbb{R}^N\times
SO(N)$. Let
$\bigtriangleup:=\{x\in\mathbb{R}^N_{+}|\sum_{\alpha}x_{\alpha}=1\}$
and for any sufficiently small $\varepsilon>0$,
$\bigtriangleup_{\varepsilon}:=\{x\in \bigtriangleup|x_{\alpha}\geq
\varepsilon, 1\leq\alpha\leq N\}$, and let $G:=\{Q\in
SO(N)|f_Q(x)\leq 0, \forall x\in \bigtriangleup\}$,
$G_{\varepsilon}:=\{Q\in SO(N)|f_Q(x)< 0, \forall x\in
\bigtriangleup_{\varepsilon}\}$. We claim that
\[G=\lim_{\varepsilon\rightarrow 0}G_{\varepsilon}=SO(N).\]
Note that this implies $(3.2)$ and thus proves the inequality. In
fact we can show
\begin{equation}
G_{\varepsilon}=SO(N) \hskip 0.5cm for \hskip 0.1cm any \hskip 0.1cm
sufficiently \hskip 0.1cm small \hskip 0.1cm \varepsilon>0.
\end{equation}
To prove $(3.3)$, we use the continuity method which consists of the
following three steps:
\\
\underline{step1}: \hskip 0.3cm $I_N\in G_{\varepsilon}$,  (thus $G_{\varepsilon}\neq\emptyset$),\\
\underline{step2}: \hskip 0.3cm $G_{\varepsilon}$ is open in $SO(N)$, \\
\underline{step3}: \hskip 0.3cm $G_{\varepsilon}$ is closed in
$SO(N)$. \vskip 0.05cm Since $F$ is uniformly continuous on
$\triangle_{\varepsilon}\times SO(N)$, step2 is obvious.\\
\underline{proof of step1}: \hskip 0.5cm  Now for any $x\in
\bigtriangleup_{\varepsilon}$,
$f_{I_N}(x)=\sum_{\alpha,\beta=1}^Nx_{\alpha}x_{\beta}\|[\hat{E}_{\alpha},
\hat{E}_{\beta}]\|^2-1$. \vskip 0.05cm It follows from $(2.2)$ that
\begin{eqnarray}
f_{I_N}(x)&=&2\{\sum_{i<j}(x_{ii}x_{ij}+x_{ij}x_{jj})+\frac{1}{2}\sum_{i<j<k}(x_{ij}x_{jk}+x_{ij}x_{ik}+x_{ik}x_{jk})\}-1\nonumber\\
&=&2\sum_{i<j}(x_{ii}x_{ij}+x_{ij}x_{jj})+\sum_{i<j<k}(x_{ij}x_{jk}+x_{ij}x_{ik}+x_{ik}x_{jk})-(\sum_{i\leq
j}^Nx_{ij})^2\nonumber\\
&<&0.\nonumber
\end{eqnarray}
which means $I_N\in G_{\varepsilon}$.\\
\underline{proof of step3}: We only need to prove the following a
priori estimate:\\ {\bf{a priori estimate}}: Suppose $f_Q(x)\leq 0$,
for every $ x\in\bigtriangleup_{\varepsilon}$. Then $f_Q(x)< 0$, for
every $
x\in\bigtriangleup_{\varepsilon}$. \\
(proof of the a priori estimate) \hskip 0.5cm If there's a point
$y\in \bigtriangleup_{\varepsilon}$ such that $f_Q(y)= 0$, without
loss of generality, we can assume $y\in
\bigtriangleup^{\gamma}_{\varepsilon}:=\{x\in\bigtriangleup_{\varepsilon}|x_{\alpha}>\varepsilon,
for\hskip 0.1cm \alpha\leq \gamma, x_{\beta}=\varepsilon, for\hskip
0.1cm \beta>\gamma\}$ for some $1\leq \gamma\leq N$. Then $y$ is a
maximum point of $f_Q(x)$ in the cone spanned by
$\bigtriangleup_{\varepsilon}$ and an interior maximum point in
$\bigtriangleup^{\gamma}_{\varepsilon}$. Therefore, there exist some
numbers $b_{\gamma+1},...,b_N$ and a number $a$ such that
\begin{equation}
\begin{cases}
(\frac{\partial f_Q}{\partial x_1}(y),...,\frac{\partial
f_Q}{\partial x_{\gamma}}(y))=2a(1,...,1),\\
(\frac{\partial f_Q}{\partial x_{\gamma+1}}(y),...,\frac{\partial
f_Q}{\partial x_{N}}(y))=2(b_{\gamma+1},...,b_N),
\end{cases}
\end{equation}
or equivalently
\begin{equation}
\sum_{\beta=1}^Ny_{\beta}(\|[\hat{Q}_{\alpha},
\hat{Q}_{\beta}]\|^2)-1=
\begin{cases}
 a,\hskip 0.7cm \alpha\leq\gamma,\\
b_{\alpha}, \hskip 0.5cm \alpha>\gamma,
\end{cases}
\end{equation}
and hence
\[
f_Q(y)=(\sum_{\alpha=1}^{\gamma}y_{\alpha})a+(\sum_{\alpha=\gamma+1}^Nb_{\alpha})\varepsilon
=0,\hskip 0.5cm
\sum_{\alpha=1}^{\gamma}y_{\alpha}+(N-\gamma)\varepsilon=1.
\]
Meanwhile, one can see that $\frac{\partial f_Q}{\partial
\nu}(y)=2(a\gamma +\sum_{\alpha=\gamma+1}^Nb_{\alpha})\leq 0$, where
$\nu=(1,...,1)$ is the normal vector of $\bigtriangleup$ in
$\mathbb{R}^N$. For any sufficiently small $\varepsilon$ (such as
$\varepsilon<\frac{1}{N}$), it follows from the above three formulas
that $a\geq 0$. Without loss of generality, we assume
$y_1=max\{y_1,...,y_{\gamma}\}>\varepsilon$ and let $J:=\{\beta\in
S| \|[\hat{Q}_{1}, \hat{Q}_{\beta}]\|^2\geq 1\}$, $n_1$ be the
number of elements of $J$.
 Now combining Lemma $2.3$ , Lemma $2.4$ and formula $(3.5)$ will give a
contradiction as following:
\begin{eqnarray}
1\leq 1+a&=&\sum_{\beta=2}^Ny_{\beta}\|[\hat{Q}_1,
\hat{Q}_{\beta}]\|^2
=\sum_{\beta\in J}y_{\beta}(\|[\hat{Q}_1,
\hat{Q}_{\beta}]\|^2-1)+\sum_{\beta\in J}y_{\beta}+\sum_{\beta\in
S/J}y_{\beta}\|[\hat{Q}_1, \hat{Q}_{\beta}]\|^2\nonumber\\
&\leq&y_1\sum_{\beta\in J}(\|[\hat{Q}_1,
\hat{Q}_{\beta}]\|^2-1)+\sum_{\beta\in J}y_{\beta}+\sum_{\beta\in
S/J}y_{\beta}\|[\hat{Q}_1,
\hat{Q}_{\beta}]\|^2\nonumber\\
&\leq&y_1+\sum_{\beta\in J}y_{\beta}+\sum_{\beta\in
S/J}y_{\beta}\|[\hat{Q}_1, \hat{Q}_{\beta}]\|^2
\leq\sum_{\beta=1}^Ny_{\beta}=1,
\end{eqnarray}
thus
\begin{equation}
y_{\beta}=y_1,\hskip 0.2cm for \hskip 0.1cm \beta\in J,\hskip 0.5cm
\sum_{\beta\in J}\|[\hat{Q}_1, \hat{Q}_{\beta}]\|^2=n_1+1\leq n<N,
\end{equation}
hence \hskip 0.3cm $S/(J\cup\{1\})\neq\emptyset$ and the last
inequality in formula $(3.6)$ should be strictly less because of the
definition of $J$ and the positivity of $y_{\beta}$ for $\beta\in
S/(J\cup\{1\})$.  \vskip 0.05cm Now we come to consider the equality
condition of Conjecture $1$ in the sight of the proof of the a
priori estimate. \vskip 0.05cm If there's a point $y\in
\bigtriangleup$ such that $f_Q(y)= 0$, without loss of generality,
we can assume $y\in
\bigtriangleup^{\gamma}:=\{x\in\bigtriangleup|x_{\alpha}>0,
\forall\alpha\leq \gamma, x_{\beta}=0, \forall\beta>\gamma\}$ for
some $1\leq \gamma\leq N$. Then $y$ is a maximum point of $f_Q(x)$
in $\mathbb{R}^N_{+}$ and an interior maximum point in
$\bigtriangleup^{\gamma}$. Therefore, we have the same conclusions
as $(3.4)$, $(3.5)$, $(3.6)$, $(3.7)$ when $\gamma=n_1+1.$ From
formula $(3.7)$, Lemma $2.1$, Lemma $2.2$ and the proof of Lemma
$2.3$, we know that all $\hat{Q}_{\beta}$ for $\beta\in J\cup\{1\}$
have the same rank 2 and $n_0=n_1=1$. Thus $\gamma=2$ and the
equality case of Lemma $2.2$ and Lemma $2.3$ tell us these two
matrices must be in the forms given in Theorem $1.1$.\\
\begin{bf}
 Proof of Corollary 1.2.
 \end{bf}
When $H\neq 0$, we can choose an orthonormal basis $\{u_1,...,u_m\}$
of $T^{\perp}_pM$ such that $u_1=\frac{H}{|H|}$. When $H=0$, the
basis $\{u_r\}$ can be chose arbitrarily. Put $B_1=A_{u_1}-|H|I_n$,
$B_r=A_{u_r}$ for $2\leq r\leq m$. Then the conclusions follow from
Theorem $1.1$ and [DFV].

\end{document}